\documentclass[a4paper,twoside,final]{amsart}       
\usepackage{latexsym,verbatim,cite,xcolor,soul}
\usepackage{amssymb,amsmath,amsthm}
\usepackage{graphicx}
\usepackage{amsfonts,amssymb,amsaddr}
\usepackage{layout,verbatim,setspace,enumerate}
\usepackage{authblk,cite}

\newtheorem{thm}{Theorem}[section]
\numberwithin{thm}{section}

\newtheorem{lem}[thm]{Lemma}

\newtheorem{rem}[thm]{Remark}
\newtheorem{prop}[thm]{Proposition}
\newtheorem{cor}[thm]{Corollary}
\newtheorem{que}[thm]{Question}
\newtheorem {ex}[thm]{Example}

\newcommand{\ve}{\varepsilon}
\newcommand{\wt}{\widetilde}

\newcommand{\ov}{\overline}
\newcommand{\mg}{\marginpar}

\newcommand{\N}{{\mathbb{N}}}
\newcommand{\R}{{\mathbb{R}}}
\newcommand{\vG}{\Gamma}
\newcommand{\mcL}{{\mathcal L}}

\newcommand{\mcI}{{\mathcal I}}

\newcommand{\mcH}{\mathcal H}

\newcommand{\vPh}{\Phi}

\begin{document}
\begin{flushright}
This work concludes a research cycle, \\but not the friendship that has tied us.\\ You left us, dear Mimmo, too soon.\\ The disease has won,  but your memories\\ will always be with us.
\end{flushright}
\begin{center}
\Large \bf Multi-integrals of finite variation
\\
\vskip.5cm \large
D. Candeloro,  L. Di Piazza,  K. Musia{\l},    
A.R. Sambucini
\vskip.5cm
\end{center}

{\bf \small \noindent Abstract}
The aim of this paper is to investigate  different types of multi-integrals of finite variation
 and to obtain decomposition results.\\



\section{Introduction}\label{intro}
  In \cite{f1994} was proved that a Banach space valued function is McShane integrable if and only if it is Pettis and Henstock integrable. That result has been then generalized to compact valued multifunctions $\Gamma$ (see \cite{dp}), weakly compact valued multifunctions (see \cite{cdpms2016b}) and bounded valued multifunctions (see \cite{cdpms2018a}). Di Piazza and  Marraffa \cite{dp-ma} presented an example of a Pettis and variationally Henstock integrable function that is not variationally McShane integrable (= Bochner integrable in virtue of \cite[Lemma 2]{dpm-ill}). It turns out that Fremlin's theorem can be formulated for variational integrals  if and only if  the variation of the integral is finite in the following sense:
$$
\sup \left\{\sum_i \left\|\int_{I_i} \Gamma  \right\|\colon \{I_1,\ldots,I_n\} \;\mbox{\rm is a finite partition of } [0,1] \right\}\, < +\infty.
$$
Finally, in the last section, using $DL$ or $Db$ conditions we are able to prove that the scalar integrability of a multifunction can be obtained as a traslation of the Pettis integrability  (Theorem \ref{p5}), while its Henstock integrability under $DL$ condition is obtained using Birkhoff integrability  (Theorem \ref{aB}),
both results with integrals of finite variation.\\
This article is the last in which Domenico Candeloro was able to cooperate and to give his personal contribution, always precious, and we want to dedicate it to him, in his memory.

\section{Preliminaria}
\label{sec:1}

Throughout   $X$ is a Banach space with norm  $\| \cdot \|$ and  its dual $X^*$.
The closed unit ball of $X$ is denoted by $B_X$. The symbol
 $c(X)$ denotes the collection of all
nonempty closed convex subsets of $X$ and $cb(X),\,cwk(X)$ and $ck(X)$  denote respectively
the family of all bounded, weakly compact and compact members of $c(X)$.  For every $C \in c(X)$ the
{\it  support function of}   $\, C$ is denoted by $s( \cdot, C)$ and
defined on $X^*$ by $s(x^*, C) = \sup \{ \langle x^*,x \rangle \colon  \ x
\in C\}$, for each $x^* \in X^*$.  $\|C\|_h= d_H(C, \{ 0\}) :=\sup\{\|x\|: x\in{C}\}$  and $d_H$ is the Hausdorff metric on the hyperspace $cb(X)$.
The map  $i:cb(X)\to \ell_{\infty}(B_{X^*})$
 given by $i(A):=s(\cdot, A)$ is the  R{\aa}dstr\"{o}m embedding (see, for example, \cite[Theorem 3.2.9 and Theorem 3.2.4(1)]{Beer},
\cite[Theorem II-19]{CV}, or \cite{L1}).\\
 $\mcI$ is the collection of all closed subintervals of the unit interval $[0,1]$.
 All functions investigated  are defined on the unit interval $[0,1]$
 endowed with Lebesgue measure $\lambda$ and Lebesgue measurable sets $\mathcal{L}$.\\
A map $\vG: [0,1]\to c(X)$ is called a {\it multifunction}.
In the sequel, given a multifunction $\vG:[0,1]\to c(X)$, we  set
$D_{\vG}(t):=\mbox{diam }(\vG(t)),$
for all $t\in [0,1]$. We say that $\vG$ satisfies the
\begin{description}
	\item[({\em $Db$-condition})] if ${\rm sup\,ess}_t D_{\vG}(t)<\infty$;	
	\item[({\em $DL$-condition})] if $\overline{\int}_0^1 D_{\vG}(t)dt<+\infty$ (where $\overline{\int}$ denotes the upper integral).
\end{description}
We recall  that a  multifunction $\vG:[0,1]\to{c(X)}$ is said to be {\it integrably bounded} if there is a  function $h\in{L_1[0,1]}$ such that $\|\vG(t)\|_h \leq |h(t)|$ for almost all $t\in[0,1]$.
We have always $D_{\vG}(t)\leq 2|\vG(t)|$. Hence, if $\vG$ is integrably bounded, then $\vG$ satisfies $DL$. If $\vG(t)\ni{0}$ for almost every $t\in[0,1]$, then $|\vG(t)|\leq D_{\vG}(t)$ a.e. Each function $g:[0,1]\to{X}$,  considered as a $ck(X)$-valued multifunction, trivially satisfies the $Db$ property. \\

We recall that if $\Phi:\mathcal{L} \to Y$ is an additive vector measure with values in a normed space $Y$,  then the {\em variation} of $\Phi$ is the extended   non negative function $|\Phi|$ whose value on a set $E \in \mathcal{L}$ is given by
$|\Phi|(E) = \sup_{\pi} \sum_{A \in \pi} \|\Phi(A)\|, $
where the supremum is taken over all partitions $\pi$ of $E$ into a finite number of pairwise disjoint members of $\mathcal{L}$. If $|\Phi| < \infty$, then $\Phi$ is called a measure of finite  variation.
\\
If the measure $\Phi$ is defined only on $\mcI$, the finite partitions considered in the definition of variation are composed by intervals. In this case we will speak of {\em  finite interval variation} and we will use the symbol $\widetilde{\Phi}$, namely:
$$
\widetilde{\Phi}([0,1])=\sup \{ \sum_i\|\Phi(I_i)\| \colon \{I_1,\ldots,I_n\}\,  \mbox{is a finite interval partition  of }[0,1]\}.
$$
If  $Y$ is a metric space, for example $(cb(X),d_H)$, which is a near vector space in the sense of \cite{L1}, and
$\Phi:\mcI\to{cb(X)}$ is additive we  consider in  its interval variation the distance $d_H (\Phi(A), \{0\})$ instead of $\|\Phi(A)\|$.\\
 We recall here briefly the definitions of integrals involved in this article.
 A scalarly integrable multifunction $\vG:[0,1]\to c(X)$ is {\it Dunford integrable} in a non-empty family $\mathcal C\subset{c(X^{**})}$,   if
for every  set $A \in \mcL$ there
exists a set $M_{\vG}^D(A)\in {\mathcal C}$  such that
\[
s(x^*,M_{\vG}^D(A))=\int_A s(x^*,\vG)\,d\lambda\,,\;\mbox{for every}\; x^*\in X^*.
\]
If $M_{\vG}^D(A)\subset{X}$ for every $A\in\mcL$, then $\vG$ is called {\it Pettis integrable}. 
 We write it as  $(P)\int_A\vG\,d\mu$ or $M_{\vG}(A)$.
	We say that a Pettis integrable $\vG: [0,1]\to c(X)$ is {\it strongly Pettis
		integrable}, if $M_{\vG}$ is an $h$-multimeasure (i.e. it is countably additive in the Hausdorff metric).
\\
 A multifunction $\vG:[0,1]\to cb(X)$ is said to be {\it Henstock} (resp. {\it McShane})
   integrable on $[0,1]$,  if there exists  $\Phi{}_{\vG}([0,1]) \in cb(X)$
    with the property that for every $\varepsilon > 0$ there exists a gauge $\delta: [0,1] \to \mathbb{R}^+$
such that for each  Perron partition (resp. partition)
   $\{(I{}_1,t{}_1), \dots,(I{}_p,t{}_p)\}$ of
   $[0,1]$ with $I_i \subset [t_i - \delta(t_i), t_i + \delta(t_i)]$ for all $i$ ( i.e. $\delta$--fine),
    we have
\begin{eqnarray}\label{e14}
d_H \left(\Phi_{\vG}([0,1]),\sum_{i=1}^p\vG(t_i)\lambda (I_i)\right)<\varepsilon.
\end{eqnarray}
If the gauges above are taken to be measurable, then we
speak of $\mathcal H$ (resp. Birkhoff)-integrability on $[0,1]$.
If $I\in\mcI$, then $\Phi_{\vG}(I):=\Phi_{\vG\chi_I}[0,1]$.
Finally if, instead of formula (\ref{e14}), we have
\begin{eqnarray}\label{var}
\sum_{i=1}^p d_H \left(\Phi_{\vG}( I_i),\vG(t_i)\lambda (I_i)\right)<\varepsilon.
\end{eqnarray}
we speak about variational {\it Henstock} (resp. {\it McShane})
   integrability on $[0,1]$. In all the cases $\Phi_{\vG}: \mcI \to cb(X)$ is an  additive interval multimeasure.\\
   Thanks to the  R{\aa}dstr\"{o}m embedding, a multifunction $\Gamma$ is "gauge" integrable (in one of the previous types) if and only if its image $i \circ \Gamma$ in $l_{\infty}(B_{X^*})$ is integrable in the same way.\\
 A multifunction $\vG:[0,1]\to{cb(X)}$ is said to be Henstock-Kurzweil-Pettis (or HKP) integrable in $cb(X)$ if it is scalarly Henstock-Kurzweil (or HK)-integrable  and for each $I\in\mcI$ there exists a set $N_{\vG}(I)\in cb(X)$ such that
 $
 s(x^*,N_{\vG}(I))=(HK) \int_Is(x^*,\vG)\quad$ for every $x^*\in{X^*}.$
 If an HKP-integrable $\vG$ is scalarly integrable, then it is called {\it weakly McShane integrable} (or wMS).

We recall that a function $f:[0,1]\to\R$ is Denjoy-Khintchine (DK) integrable (\cite[Definition 11]{gordon-p}), if there exists an ACG function (cf. \cite{Gor})  F such that its approximate derivative is almost everywhere equal to $f$. \\
A multifunction $\vG:[0,1]\to{cb(X)}$ is Denjoy-Khintchine-Pettis (DKP)
 integrable in
 a non empty family $\mathcal C$ in $cb(X)$, if for each $x^*\in{X^*}$ the function $s(x^*,\cdot)$ is Denjoy-Khintchine integrable and for every $I\in\mcI$ there exists $C_I\in\mathcal C$ with $(DK)\int_Is(x^*,\vG)=s(x^*,C_I)$, for every $x^*\in{X^*}$.
\\
As regards other definitions
 of measurability and  integrability  that will be treated here and are not explained and the known relations among them, we refer to \cite{cdpms2016,cdpms2016a,cdpms2016b,cdpms2019a,cs2014,ccgs,ccgs2019,ckr2,eh,dms,nara},    in order do not burden the presentation.

\section{Multimeasures of finite variation}\label{sec:2}
We begin with a known fact.
\begin{lem}\label{l1}
If $f:[0,1]\to\R$ is the Denjoy-Khintchine integrable and the  interval variation
of its integral is finite,  then $f$ Lebesgue integrable.
\end{lem}
\begin{proof}
Let $F$ be the Denjoy-Khintchine  primitive of
$f: [a,b] \to \R$. Then $F$ is an ACG function  and, according to  \cite[Theeorem 15.8]{Gor},
$F$ is continuous on $[a,b]$. So $F$ satisfies the condition (N)  of Lusin on in $[a,b]$ (see  \cite[Theorem 6.12]{Gor}). Since $F$ is also BV,  an application of  \cite[Theorem 6.15]{Gor} gives that $F$ is also AC on on $[a,b]$. So $f$ is Lebesgue integrable.\qed
\end{proof}

\begin{thm}\label{t1}
 Let $\Phi:\mcI\to{cb(X)}$ be the DKP-integral
 of $\vG:[0,1]\to{cb(X)}$. If $\sup_{x^*\in{B_X}}\widetilde{\langle{x^*,\Phi}\rangle}([0,1])<\infty$, then $\vG$ is weakly McShane integrable in $cb(X)$ and Gelfand integrable in $cw^*k(X^{**})$. If $\wt\Phi([0,1])<\infty$, then $\Phi$ is strongly Pettis integrable in $cb(X)$.
\end{thm}

\begin{proof}
By Lemma \ref{l1}  $\vG$ is wMS-integrable in $cb(X)$. According to \cite[Theorem 3.2]{cdpms2018a} it is Gelfand integrable in $cw^*k(X^{**})$. Denote the Gelfand integral by $\Psi$.  \\
Assume now that $\wt\Phi ([0,1])<\infty$.
 If $\{I_i : i \in \mathbb{N} \}$ is a sequence of non-overlapping subintervals of $[0,1]$, then
$$
\sum_i\|\Phi(I_i)\|_h \leq \wt \Phi([0,1])<\infty
$$
and so, due to the completeness of $cb(X)$ under Hausdorff distance,
 the series $\sum_i\Phi(I_i)$ is convergent in $cb(X)$.\\
 But for each $x^*\in{X^*}$ the function $s(x^*,\Psi)$ is a measure and so
$\sum_is(x^*,\Phi(I_i))=s(x^*,\Psi(\bigcup_iI_i))$. Since the sum of $\sum_i\Phi(I_i)$ is uniquely determined, we have
 $$
  \Psi(\bigcup_i I_i)=\sum_i\Phi(I_i)\in{cb(X)}\,.
    $$

  It follows that $\Psi$ is $\sigma$-additive (in the Hausdorff metric) on the algebra $\mathfrak J$ generated by $\mcI$. Hence, also $i\circ\Psi$ is $\sigma$-additive on 
  $\mathfrak J$.
  But  $\wt{i\circ\Psi}([0,1])=\wt{\Psi}([0,1])=\wt\Phi([0,1])<\infty$ and so due to \cite[Proposition I.15]{du}, $i\circ\Psi$ restricted to $\mathfrak J$ is  strongly additive.  \\
It is a consequence of \cite{katz} or \cite[Theorem I.5.2]{du} that $i\circ\Psi$ is a measure on the $\sigma$-algebra of Borel subsets of $[0,1]$. But  $i\circ\Psi(E)=0$, provided  Lebesgue measure vanishes on $E$ and consequently, $i\circ\Psi$ is measure on $\mcL$. Since $i(cb(X))$ is a closed cone  also $\Psi$ is a measure in the Hausdorff metric of $cb(X)$ and therefore $\vG$ is strongly Pettis integrable on $\mcL$.
\qed
\end{proof}

\begin{cor}
If $\vG\colon [0,1]\to{c(X)}$ is Pettis integrable in $cb(X)$, $M_{\Gamma}$ is its indefinite Pettis integral and $|M_{\Gamma}|([0,1])<\infty$, then $\vG$ is strongly Pettis integrable.
\end{cor}
\begin{proof} It is easily seen that due to the finite variation of $M_{\Gamma}$, the multifunction $\vG$ takes a.e. bounded values. Without loss of generality we may assume that $\vG:[0,1]\to{cb(X)}$.
 We have $\wt{M_{\Gamma}}([0,1])\leq |M_{\Gamma}|([0,1])<\infty$ and so we may apply Theorem \ref{t1}. \qed
 \end{proof}

 Under stronger assumptions one obtains stronger results. We proved in \cite{cdpms2018a} the following
 \begin{thm}\label{T11}
Let $\vG:[0,1]\to{cb(X)}$ be Henstock (or $\mcH$) integrable and let $\vPh_{\vG}$ be its Henstock ($\mcH$)-integral. If $\wt{\vPh}_{\vG}[0,1]<\infty$, then $\vG$ is McShane (Birkhoff) integrable.
\end{thm}
Finally, we can formulate the characterization of variationally McShane integral in terms of the variational Henstock integral.
\begin{thm}\label{t2}
  A multifunction $\vG:[0,1]\to{cb(X)}$ is variationally McShane integrable
   if and only if it is variationally Henstock integrable and  the interval variation of the Henstock integral is finite.
\end{thm}

\begin{proof}
 We need to prove only that each  vH-integrable multifunction $\vG:[0,1] \to {cb(X)}$ with integral of finite  interval variation is
 variationally McShane  integrable.
 We know already from Theorem \ref{t1} that $\vG$ is Pettis integrable. Since  $i\circ\vG$ is vH-integrable it is strongly measurable. If $M_{\vG}$ is the Pettis integral of $\vG$, then $i\circ{M_{\vG}}$ is a measure of finite variation and $ i\circ{M_{\vG}}(I)=(vH)\int_Ii\circ\vG$. It follows that $i\circ\vG$ is  Bochner integrable. Now we may apply
\cite[Proposition 3.6]{cdpms2016} to obtain variational McShane integrability of $\vG$. \qed
\end{proof}

In case of  vector valued functions $f: [0,1] \to X$,  by the properties of the Pettis and the Bochner integrals, it follows at once  that if $f$ is strongly measurable,  Pettis integrable and  its Pettis integral has  finite variation, then $f$ is Bochner integrable.  The next result is the multivalued version of  this result.

\begin{thm}\label{new}
Let $\vG: [0,1] \to cb(X)$ be
Bochner measurable, Pettis integrable, and  its Pettis integral has  finite variation. Then
  $\vG$ is  integrably bounded.
\end{thm}
\begin{proof} Since $\vG$ is Bochner measurable, it is a.e. limit of simple multifunctions.   It follows that  $i\circ\vG$ is strongly measurable. Let us assume that $Y:=\ov{span}(i\circ{\vG([0,1])})$ is a closed separable subspace of $ \ell_{\infty} (B_{X^*})$. Then, we follow the proof of \cite[Proposition 3.5]{ckr0}.
 If $e_{x^*}\in{B_{\ell_{\infty}(B_{X^*})^*}}$ is defined by $\langle{e_{x^*},g}\rangle:=g(x^*)$ for every $g\in\ell_{\infty}(B_{X^*})$, then the set $B:=\{e_{x^*}|Y:x^*\in{B_{X^*}}\}\subset B_{Y^*}$ is norming. By the Pettis integrability of $\vG$
the family ${\mathcal Z}_{i\circ\vG,B}:=\{\langle{e_{x^*},i\circ\vG}\rangle: x^*\in{B_{X^*}}\}=\{s(x^*,\vG):x^*\in{B_{X^*}}\}$ is uniformly integrable.
 Consequently, $i\circ\vG$ is a Pettis integrable function. Moreover,
$i((P)\int_A\vG\,d\lambda)=(P)\int_Ai\circ\vG\,d\lambda$ for every $A\in\mcL$ (see the proof of \cite[Proposition 3.5]{ckr0}). By the assumption the variaton of $(P)\int{i\circ\vG\,d\lambda}$ is finite and so $i\circ\vG$ is Bochner integrable. Consequently, $\vG$ is integrably bounded. \qed
\end{proof}
Then by  \cite[Proposition 3.6]{cdpms2016} (formulated for $cb(X)$-valued multifunctions)
and Theorem \ref{new} we get the following
\begin{prop}
 Let $\Gamma:[0,1]\to{cb(X)}$ be a scalarly measurable multifunction. Then the following conditions are equivalent:
 \begin{enumerate}
 \item
 $\Gamma$ is variationally McShane integrable;
  \item
  $i(\Gamma) \in  L_1(\lambda, \ell_{\infty} (B_{X^*}))$;
  \item
  $\Gamma$ is  Bochner measurable  and integrably bounded;
  \item $\Gamma$ is  Bochner measurable, Pettis integrable, and  its Pettis integral has  finite variation.
  \end{enumerate}
\end{prop}
\begin{proof}
 It is an immediate consequence  of  Theorem \ref{new} if we proceed analogously to  \cite[Proposition 3.6]{cdpms2016}.
\qed
\end{proof}
\section{Decompositions}
In the study of the integrability of multifunctions it is important to decompose a multifunction as a sum of a selection that is integrable in the same sense and a multifunction that is integrable in a stronger sense than the original one
 (see for example \cite{ft,dpm-ill,dm2,cdpms2016,cdpms2016a,cdpms2016b,cdpms2018a}).
Using $Db$ or $DL$ conditions we are able to extend
 decomposition results and to write integrable multifunctions as a translation of a multifunction with its integral of finite variation.
\begin{thm}\label{p5}
Let $\vG:[0,1]\to{c(X)}$ be integrable in $cb(X)$ ($cwk(X)$ or $ck(X)$)  in the sense of one of the scalarly defined integrals. If $\vG$ possesses at least one selection integrable in the same way, then the following conditions are equivalent:
\begin{enumerate}
\item
$\vG$ satisfies the $DL$-condition (or $Db$ condition);
\item
$\vG=G+f$, where $f$ is a properly integrable selection of $\vG$, $G$ is Pettis integrable in $cb(X)$ ($cwk(X)$ or $ck(X)$) and $\ov{\int}_0^1D_G(t)\,dt<\infty$ ( and $G$ is bounded). In particular the indefinite integral of $G$ is of finite variation.
\end{enumerate}
\end{thm}
\begin{proof}
 Assume that $\vG$ is DP-integrable. Due to \cite[Theorem 3.5]{cdpms2018a} $\vG=G+f$, where $G$ is Pettis integrable, $f$ is Denjoy integrable and  $G$ satisfies the condition DL.  It is obvious that the Pettis integral of $G$ is of finite variation.
\qed
\end{proof}
\begin {rem}\label{r1}
\rm
Unfortunately, even if $G:[0,1]\to{ck(X)}$ is a positive multifunction that is Pettis integrable and its integral is of finite variation, the multifunction $G$ may not satisfy the DL condition. To see it let $X=\ell_2[0,1]$ and let $\{e_t:t\in(0,1]\}$ be its orthonormal system. If $G(t):=conv\{0,{e_t}/t\}$, then $s(x,G)=0$ a.e. for each separate $x\in\ell_2[0,1]$ and so the integral and its variation are equal  zero. However, $diam\{G(t)\}=1/t$ and so the DL-condition fails.
Moreover, $G$ is  not Henstock integrable.
 Indeed, let $\delta$ be any gauge and $\{(I_1,t_1),\ldots, (I_n,t_n)\}$ be a $\delta$-fine Perron partition of $[0,1]$.
 Assume that $0\in{I_1}$, then $t_1\leq|I_1|$. Hence $\lambda(I_1)/{t_1}\geq1$ for $t_1>0$ and so
 $
  \left\|\sum_{i\leq{n}} \frac{e_i}{t_i}\lambda(I_i)\right\|\geq 1.
 $
 Consider now the multifunction given by $H(t):=conv\{0,e_t\}$, where $X$ is as above. We are going to prove that $H$ is Birkhoff-integrable. Given $\ve>0$, let $n\in\N$ be such that $1/{\sqrt{n}}<\ve$ and $\delta$ be any gauge, pointwise less than $1/n$. If $\{(I_1,t_1),\ldots,(I_m,t_m)\}$ is a $\delta$-fine  partition of $[0,1]$ and $\{J_1,\ldots,J_n\}$ is the division of $[0,1]$ into closed intervals of the same length, then
\begin{eqnarray*}
 \biggl \|\sum_{i\leq{m}}e_i \lambda(I_i) \biggr\|&=& \biggl\|\sum_{i\leq{m}}\sum_{k\leq{n}}e_i \lambda(I_i\cap{J_k})\biggr\|=\biggl\|\sum_{k\leq{n}}\sum_{i\leq{m}}e_i \lambda(I_i\cap{J_k})\biggr\|\\
 &=&
  \biggl(\sum_{k\leq{n}}\sum_{i\leq{m}} \lambda(I_i\cap{J_k})^2\biggr)^{1/2}\leq1/{\sqrt{n}}<\ve\,.
\end{eqnarray*}
(We apply here the inequality $\sum_{i}a_i^2\leq (\sum_{
}a_i)^2$. For each fixed $k\leq{n}$ we take as $a_i$ the number $\lambda(I_i\cap{J_k})$
). If $\delta$ is measurable, then we get Birkhoff integrability of $H$.
\end{rem}
Some additional results will be given now, in order to get decompositions with gauge integrable multifunctions.
\begin{thm}\label{aB}
Let $\vG:[0,1]\to cwk(X)$ satisfy  $DL$-condition, and assume that $\vG$ is  $\mathcal{H}$-integrable (or H-integrable). Then we have
$\vG=G+f,$
where $f\in\mathcal{S_H}(\vG)$ ($f\in\mathcal{S}_H(\vG)$) is arbitrary and $G$ is an abs-Birkhoff integrable multifunction.
 In particular the integral of $G$ has finite variation.
 If $\vG$ is Bochner measurable, then $G$ is also variationally Henstock integrable.
\end{thm}
\begin{proof}
  Assume that $\vG$ is $\mathcal{H}$-integrable. It is known (see \cite[Theorem 3.1]{dp}) that $\vG$ has an $\mathcal{H}$-integrable selection $f$. Thanks to \cite[Theorem 4]{nara}, both $i \circ \vG$ and $f$ are Riemann-measurable. So, if $G:=\vG-f$, it is clear that $i \circ G$ is Riemann-measurable too. Moreover, thanks to the $DL$-condition, the function $t\mapsto \|i \circ G(t)\|$ is integrably bounded, i.e.
$\overline{\int}_0^1  \|i \circ G(t)\| dt<+\infty$
since
$\|i \circ G(t)\|=\sup\{\|u\|_X: u\in G(t)\}=\sup\{\|v-f(t)\|_X: v\in \vG(t)\}\leq \mbox{diam} (\vG(t))$.
So, $i \circ G$ is Riemann-measurable and integrably bounded, which means that $i \circ G$ (and so $G$) is absolutely Birkhoff integrable, thanks to \cite[Theorem 2]{ncm}.
\\
 Assume now that $\vG$ is $H$-integrable. Then, according to  \cite[Theorem 3.5]{cdpms2018a}  $\vG=G+f$, where $G$ is Birkhoff integrable. By the assumption $G$ satisfies the DL-condition. Hence again the function $t\mapsto \|i \circ G(t)\|$ is integrably bounded. Consequently, $i\circ{G}$ is absolutely Birkhoff integrable and hence also $G$. The vH-integrability of $G$ follows from \cite[Corollary 4.1]{bdm}, since $G$ is Pettis integrable.
\qed \end{proof}

A similar result can be given also for Birkhoff integrable functions $\vG:[0,1]\to{cwk(X)}$: the proof is essentially the same but instead of \cite{dp} we invoke  \cite[Theorem 3.4]{cdpms2016}.
\begin{prop}\label{df1}
Let $\vG:[0,1]\to cwk(X)$ satisfy  $DL$-condition, and assume that $\vG$ is Birkhoff integrable. Then we have
$\vG=G+f,$
where $f$ is any Birkhoff integrable selection of $\vG$, and $G$ is an abs-Birkhoff integrable multifunction.  In particular the integral of $G$ has finite variation.
\end{prop}

\begin{que} \rm
Assume that $f:[0,1]\to{X}$ is Birkhoff integrable and the classical variation of the indefinite integral is finite. Is $f$ absolutely Birkhoff integrable? That is, do we have
 $\ov{\int}_0^1\|f(t)\|\,dt<\infty$?
A partial answer is  contained in \cite[Corollary 2]{ncm}. \\
Another way, does there exist a Birkhoff integrable $f$ that is scalarly equivalent to zero and $\ov{\int}_0^1\|f(t)\|\,dt=\infty$?
Recall that $G$ from Remark \ref{r1} is not Birkhoff integrable. \\
Fremlin proved that a Birkhoff integrable function is properly measurable in the Talagrand sense. It is known that $f$ is Talagrand integrable if and only if $f$ is properly measurable and  $\ov{\int}\|f\|\,d\lambda<\infty$. Then, it is known that $f$ is absolutely Birkhoff integrable  if and only if it is Riemann measurable and $\ov{\int}\|f\|\,d\lambda<\infty$ if and only if $f$ is Birkhoff integrable and $\ov{\int}\|f\|\,d\lambda<\infty$. Thus, if $f$ is  absolutely Birkhoff integrable, then $f$ is also Talagrand integrable. The converse result fails by \cite[Example 3C]{FM} (where a function $f: [0,1] \to \ell_{\infty}(\mathbb{N})$ is shown, which is Talagrand but not even McShane integrable).
\end{que}
It is also possible to obtain decompositions where the multifunction $G$ turns out to be   variationally McShane integrable, as follows.
\begin{prop}\label{df2}
Let $\vG:[0,1]\to cwk(X)$ satisfy  $DL$-condition, and assume that $\vG$ is  Bochner measurable. Then we have
$\vG=G+f,$
where $f$ is any strongly measurable selection of $\vG$, and $G$ is a   variationally McShane integrable multifunction.
\end{prop}
\begin{proof}
 Let $f$ be any strongly measurable selection of $\vG$, and set $G=\vG-f$. Then clearly $G$ is Bochner measurable. Moreover, since $\vG$ satisfies the $DL$ condition and $f$ is  a selection from $\vG$, $i \circ G$ is integrably bounded. Then $i \circ G$ is strongly measurable and integrably bounded, and therefore  variationally McShane integrable.
Of course this implies that also $G$ is  integrable.
\qed
\end{proof}

\begin{prop}\label{p9}
If $\vG: [0,1] \to cwk(X)$ is Henstock ($\mathcal H$, variationally Henstock, Pettis, McShane,  Birkhoff) integrable, then $\Gamma$ cannot be, in general,  written as $G + f$, where $G $ is   variationally McShane integrable and $f$ is integrable in the same way as $\vG$.
\end{prop}
\begin{proof}
Take $f$ as in \cite{dp-ma}; then $f$ is vH and Pettis integrable, but not Bochner integrable. Let $\vG = {\rm conv} \{0, f(t) \}$, then $\vG$ is vH-integrable, Pettis but not  Bochner integrable, as shown in \cite[Example 4.7]{cdpms2016}. By \cite[Theorem 4.3 (a) and (c)]{cdpms2016}, then $\vG$ is also McShane integrable  and then Birkhoff integrable, since $\vG$ is Bochner measurable. Following the same motivations of \cite[Remark 5.4]{cdpms2016b} the multifunction $G = \vG - f $ is not    variationally McShane integrable.
\qed
\end{proof}
Almost nothing is known on possible decompositions of a Pettis integrable multifunctions. We have only the following negative result:
\begin {ex}\label{ex17}
\rm
Let $\vG:[0,1]\to cwk(X)$ be Pettis integrable. Assume that $M_{\vG}(\mcL)$ is not relatively compact in the Hausdorff metric. Then $\vG$ cannot be represented as
$\vG=G+f$, where $G$ is McShane integrable and
$f$ is a Pettis integrable selection of $\vG$. In such a case $i\circ\vG$ is not Pettis integrable.
\end{ex}
\begin{proof}
 Suppose that such a decomposition exists. Then, since $G$ is McShane integrable, the function $i \circ G$  is also McShane integrable and consequently it has relatively norm compact range.That is however equivalent to the norm relative compactness of   $M_G(\mcL)$ in $d_H$.  But then $G$ can be approximated by simple functions (see \cite[Theorem 2.3]{mu8}). Since the integral of $f$ is norm relatively compact (because Lebesgue measure is perfect) also $f$ can be approximated by simple functions in the Pettis norm (see \cite[Theorem 9.1]{mu}).  As a result the multifunction $\vG$ can be approximated by simple multifunctions, which is impossible, since its range  $M_{\vG}(\mcL)$ is not relatively compact in the Hausdorff metric. (see  \cite[Theorem 2.3]{mu8}). The non-integrability of $i\circ\vG$ is a consequence of perfectness of Lebesgue measure. Indeed, the range of the integral of a Pettis integrable function on $[0,1]$ (or on any perfect measure space) is norm relatively compact (\cite[3J]{ft}).
\qed
\end{proof}

\section*{Acknowledgements}
This research was partially supported by Grant  "Metodi di analisi reale per l'appros\-simazione attraverso operatori discreti e applicazioni” (2019) of GNAMPA -- INDAM (Italy), by University of Perugia --  Fondo Ricerca di Base 2019 ``Integrazione, Approssimazione, Analisi Nonlineare e loro Applicazioni``, by University of Palermo  --  Fondo Ricerca di Base 2019 and by Progetto Fondazione Cassa di Risparmio cod. nr. 2018.0419.021 (title: Metodi e Processi di Intelligenza
artificiale per lo sviluppo di una banca di immagini mediche per fini diagnostici (B.I.M.)).

\small

Domenico Candeloro and Anna Rita Sambucini,
               Department of Mathematics and Computer Sciences, University of  Perugia 1, Via Vanvitelli - 06123, Perugia (Italy),
               orcid ID: 0000-0003-0526-5334,  0000-0003-0161-8729
              \email{candelor@dmi.unipg.it, anna.sambucini@unipg.it}\\

   Luisa Di Piazza, 
    Department of Mathematics and Computer Sciences, University of Palermo,  Via 	Archirafi 34, 90123 Palermo (Italy)  orcid ID: 0000-0002-9283-5157				  \email{luisa.dipiazza@unipa.it}
\\

	Kazimierz Musia\l,
	Institute of Mathematics, Wroc{\l}aw University,  Pl. 			Grunwaldzki  2/4, 50-384 Wroc{\l}aw (Poland) orcid ID: 0000-0002-6443-2043
	\email{kazimierz.musial@math.uni.wroc.pl}

\end{document}